\nonstopmode
\documentclass[10pt]{amsart}
\usepackage{graphicx}
\usepackage{latexsym}
\usepackage{fancyhdr}
\usepackage{amsmath, amssymb, amsthm}
\usepackage[all]{xy}
\usepackage{pdflscape}
\usepackage{longtable}
\usepackage{rotating}
\usepackage{verbatim}
\usepackage{hyperref}
\hypersetup{
    linkcolor = blue,
    citecolor = blue,
    urlcolor  = blue,
    colorlinks = true,
}

\usepackage{cleveref}
\usepackage{subfigure}
\usepackage{mathrsfs}
\usepackage{mdwlist}
\usepackage{dsfont}
\usepackage{mathtools}
\usepackage{float}
\usepackage{color}
\usepackage{stmaryrd}

\usepackage{tkz-base}
\usepackage{pgfplots}
\pgfplotsset{compat=newest}

\usepackage{tkz-euclide}
\usepackage{etoolbox}

\definecolor{teal}{rgb}{0.0, 0.5, 0.5}

\newcounter{mnotecount}[section]

\newcommand{\rmnote}[1]{}

\allowdisplaybreaks

\DeclareFontFamily{U}{mathb}{\hyphenchar\font45}
\DeclareFontShape{U}{mathb}{m}{n}{
      <5> <6> <7> <8> <9> <10> gen * mathb
      <10.95> mathb10 <12> <14.4> <17.28> <20.74> <24.88> mathb12
      }{}
\DeclareSymbolFont{mathb}{U}{mathb}{m}{n}
\DeclareFontSubstitution{U}{mathb}{m}{n}

\let\dot\relax
\DeclareMathAccent{\dot}{0}{mathb}{"39}
\let\ddot\relax
\DeclareMathAccent{\ddot}{0}{mathb}{"3A}
\let\dddot\relax
\DeclareMathAccent{\dddot}{0}{mathb}{"3B}
\let\ddddot\relax
\DeclareMathAccent{\ddddot}{0}{mathb}{"3C}

\theoremstyle{plain}
\newtheorem*{theorem*}{Theorem}
\newtheorem{theorem}{Theorem}[section]
\newtheorem*{lemma*}{Lemma}
\newtheorem{lemma}[theorem]{Lemma}
\newtheorem*{assumption*}{Assumption}

\newtheorem*{proposition*}{Proposition}
\newtheorem{proposition}[theorem]{Proposition}
\newtheorem*{corollary*}{Corollary}
\newtheorem{corollary}[theorem]{Corollary}
\newtheorem*{claim*}{Claim}

\newtheorem*{conjecture*}{Conjecture}

\newtheorem*{question*}{Question}
\theoremstyle{definition}
\newtheorem*{definition*}{Definition}

\newtheorem*{example*}{Example}
\newtheorem{example}[theorem]{Example}

\newtheorem*{algorithm*}{Algorithm}
\newtheorem*{remark*}{Remark}
\newtheorem*{remarks*}{Remarks}
\newtheorem{remark}[theorem]{Remark}

\newtheorem*{convention*}{Convention}

\theoremstyle{plain}
\newtheorem{theoremA}{Theorem}

\def\al{\alpha}
\def\be{\beta}

\def\de{\delta}
\def\ep{\epsilon}

\def\th{\theta}

\def\si{\sigma}

\def\vh{\varphi}
\def\ch{\chi}
\def\ps{\psi}
\def\om{\omega}

\def\Ga{\Gamma}

\def\Th{\Theta}

\def\Si{\Sigma}

\def\Ph{\Phi}

\def\C{\mathbb{C}}

\def\K{\mathbb{K}}
\def\N{\mathbb{N}}

\def\R{\mathbb{R}}
\def\Z{\mathbb{Z}}

\def\cC{\mathcal{C}}

\def\cP{\mathcal{P}}

\def\sC{\mathscr{C}}

\def\sU{\mathscr{U}}

\def\p{\partial}

\renewcommand{\Re}{\mathrm{Re}}

\def\<{\langle}
\def\>{\rangle}
\renewcommand{\o}{\circ}

\def\ol{\overline}

\let\on=\operatorname

\newcommand{\sr}[1]%
{\ifmmode{}^\dagger\else${}^\dagger$\fi\ifvmode
\vbox to 0pt{\vss
 \hbox to 0pt{\hskip\hsize\hskip1em
 \vbox{\hsize3cm\raggedright\pretolerance10000
 \noindent #1\hfill}\hss}\vss}\else
 \vadjust{\vbox to0pt{\vss%
 \hbox to 0pt{\hskip\hsize\hskip1em%
 \vbox{\hsize3cm\raggedright\pretolerance10000%
 \noindent #1\hfill}\hss}\vss}}\fi%
}

\providecommand{\mapsfrom}{\kern.2em%
\setbox0=\hbox{$\leftarrow$\kern-.10em\rule[0.26mm]{0.1mm}{1.3mm}}\box0%
\kern.3em}

\title[On real analytic functions on closed subanalytic domains]
{On real analytic functions on closed subanalytic domains}

\author[A.~Rainer]{Armin Rainer}

\address{Fakult\"at f\"ur Mathematik, Universit\"at Wien,
Oskar-Morgenstern-Platz~1, A-1090 Wien, Austria}
\email{armin.rainer@univie.ac.at}

\begin{document}

\begin{abstract}
We show that a function $f : X \to \R$ 
defined on a closed uniformly polynomially cuspidal set $X$ in $\R^n$
is real analytic if and only if $f$ is smooth and all its composites with germs of polynomial curves 
in $X$ are real analytic.
The degree of the polynomial curves needed for this 
is effectively related to the regularity of the boundary of $X$.
For instance, if the boundary of $X$ is locally Lipschitz, then 
polynomial curves of degree $2$ suffice.
In this Lipschitz case, we also prove that a function $f : X \to \R$ 
is real analytic if and only if all its composites with germs of quadratic polynomial maps in two variables with images in $X$ 
are real analytic; here it is not necessary to assume that $f$ is smooth.
\end{abstract}

\thanks{Supported by FWF-Project P 32905-N}
\date{\today}

\maketitle

\section{Introduction}

In this note, we are interested in Hartogs-type characterizations of real analytic functions.
Let us recall two fundamental results due to Bochnak and Siciak: 
let $U \subseteq \R^n$ be a nonempty open set and $f : U \to \R$ any function.

\begin{theoremA}[{\cite{BochnakSiciak71,Siciak70}}] \label{res:A}
The function $f$ is real analytic if and only if $f$ is smooth and the restriction of $f$ to each affine line 
        that meets $U$ is real analytic.
\end{theoremA}

\begin{theoremA}[{\cite{Bochnak:1971aa,Bochnak:2018aa}}] \label{res:B}
The function    $f$ is real analytic if and only if the restriction of $f$ to each affine $2$-plane that meets $U$ is real analytic.
\end{theoremA}

In Theorem \ref{res:A}, the assumption that $f$ is smooth cannot be omitted.
Recently, Bochnak, Koll\'ar, and Kucharz \cite{Bochnak:2020tz} proved a global version of Theorem \ref{res:B}:  
a function $f : M \to \R$ on a real analytic manifold $M$ of dimension $n \ge 3$ is real analytic 
if $f|_N$ is real analytic for every real analytic submanifold $N \subseteq M$ 
that is homeomorphic to the $2$-sphere. 

We will investigate versions of these results on closed fat subsets $X$ of $\R^n$ with cusps 
(that $X$ is \emph{fat} means that it is contained in the closure of its interior, i.e., $X =\ol{X^\o}$).
Even if $X$ has Lipschitz boundary, we will have to 
compose $f$ with germs of quadratic polynomial maps 
instead of just affine maps. 
It will turn out that the maximal degree of the polynomial maps needed to 
detect real analyticity is strongly related to the regularity of the boundary (i.e., to the sharpness of the cusps).

A recent result of Kucharz and Kurdyka \cite{Kucharz:2023aa} (see \Cref{res:C} below)
shows that, for subanalytic functions on real analytic manifolds, 
real analyticity can be recognized by restriction to real analytic subsets of dimension one. 
We will discuss a variant on suitable closed fat subsets $X$ of $\R^n$.

This note is a natural continuation of our papers \cite{Rainer18} and 
\cite{Rainer:2021tr}. 
Before we can state the results, some terminology must be introduced.

\subsection{Plot-analytic functions}

Let $d,m,n$ be positive integers. 
Let $X \subseteq \R^n$ be nonempty.
By an \emph{$(m,d)$-plot in $X$} we mean the germ $[p]$ at $0$ of a polynomial map $p  = (p_1,\ldots,p_n): \R^m \to \R^n$ 
\begin{itemize}
    \item of degree at most $d$, i.e.,  
        $p_i \in \R[x_1,\ldots,x_m]$ and $\max_{1 \le i \le n}\deg p_i \le d$,
\item and image contained in $X$, i.e., there exists a neighborhood $U$ of $0$ in $\R^m$ such that 
    $p(U) \subseteq X$.
\end{itemize}
We will write $p$ instead of $[p]$; this slight abuse of notation will 
lead to no confusion.

Let $\cP_{m,d}(X)$ denote the set of all $(m,d)$-plots in $X$.  
We will also consider 
\[
    \cP_m(X) := \bigcup_{d =1}^\infty \cP_{m,d}(X)
\]
and call its elements \emph{polynomial $m$-plots in $X$}. 

For any function $f : X \to \R$ and any $p \in \cP_m(X)$
it is meaningful to consider the function germ $f \o p$ at $0$ in $\R^m$.
We define $\sC^\om_{m,d}(X)$ to be the set of all functions $f : X \to \R$ such that 
\begin{equation*}
    f \o p \text{ is a real analytic germ at } 0 \in \R^m, \text{ for all } p \in \cP_{m,d}(X).
\end{equation*}
Furthermore,  we set
\[
    \sC^\om_{m}(X) := \bigcap_{d=1}^\infty \sC^\om_{m,d}(X).
\]
We call the elements of $\sC^\om_{m,d}(X)$ \emph{$(m,d)$-plot-analytic functions}. 
Recall that $f : X \to \R$ is called \emph{arc-analytic} if $f \o c$ is real analytic for each germ $c$ of a real analytic arc in $X$.
The arc-analytic functions on $X$ form a subset of $\sC^\om_1(X)$.

Let $X \subseteq \R^n$ be closed.
Let $\cC^\om|_X$ (resp.\ $\cC^\infty|_X$) be the set of all functions $f : X \to \R$
such that there exist an open neighborhood $U$ of $X$ in $\R^n$ 
and a real analytic (resp.\ smooth) function $F : U \to \R$ such that $F|_X=f$.
Note that each $f \in \cC^\om|_X$ is the restriction to $X$ of a holomorphic function defined on an open neighborhood of $X$ in $\C^n$.

\subsection{UPC sets}

Let us recall (cf.\ \cite{PawluckiPlesniak86}) that 
a closed set $X \subseteq \R^n$ is called \emph{uniformly polynomially cuspidal} (UPC) 
if there exist positive integers $m,D$ and a constant $M>0$ such that for each $x \in X$ there is a polynomial curve $h_x : \R \to \R^n$
of degree at most $D$ 
such that 
\begin{enumerate}
    \item $h_x(0)=x$,
    \item $\on{dist}(h_x(t),\R^n \setminus X) \ge M\, t^m$ for all $x \in X$ and $t \in [0,1]$. 
\end{enumerate}
In that case, we say that the UPC set $X$ has the \emph{characteristic} $(m,D)$. (The constant $M$ will not be  important for us.)
Let $\on{char}(X)$ be the set of all characteristics of $X$.
Note that if $(m,D) \in \on{char}(X)$ then $(m,D) + \N^2 \subseteq \on{char}(X)$.
We define 
\begin{equation*}
    d(X) := 2 \min_{(m,D) \in \on{char}(X)} \max\{m,D\}.
\end{equation*}

Any UPC set $X$ is fat. 
We say that $X$ is \emph{simple} if each $x \in \p X$ 
has a basis of neighborhoods $\sU$ such that $U \cap X^\o$ is connected for all $U \in \sU$.
We shall see in \Cref{ex:simple} that, regarding our results, to be simple is a natural and indispensable condition.

An important class of UPC sets is the collection of all compact fat subanalytic sets $X \subseteq \R^n$, 
see Paw{\l}ucki and Plesniak \cite{PawluckiPlesniak86}.

\subsection{H\"older and Lipschitz sets}
Let $\al \in (0,1]$ and $r,h >0$. The set 
\[
   \Ga^\al_n(r,h) := \Big\{(x',x_n) \in \R^{n-1}\times \R : |x'|<r, \, \big(\tfrac{|x'|}{r}\big)^\al < \tfrac{x_n}{h} < 1\Big\}
\]
is a \emph{truncated open $\al$-cusp}.
By an \emph{$\al$-set} we mean a closed fat set $X \subseteq \R^n$ 
such that $X^\o$ has the \emph{uniform $\al$-cusp property}: 
for each $x \in \p X$ there exist $\ep>0$, a truncated open $\al$-cusp $\Ga = \Ga^\al_n(r,h)$,
and an orthogonal linear map $A \in O(n)$ such that $y + A\Ga \subseteq X^\o$ for all 
$y \in X \cap B(x,\ep)$. 
A bounded open set in $\R^n$ has the uniform $\al$-cusp property if and only if it has $\al$-H\"older boundary; cf.\ \cite[Remark 2.1]{Rainer:2021tr}. 
We say that $X$ is a \emph{H\"older set} if $X$ is an $\al$-set for some $\al \in (0,1]$; 
$1$-sets are also called \emph{Lipschitz sets}. Note that H\"older sets are always simple; see 
\cite[Proposition 3.9]{Rainer18}.

By definition, any compact $\al$-set $X$ is a UPC set of characteristic $(\lceil \al^{-1} \rceil,1)$
so that $d(X) \le 2 \lceil \al^{-1} \rceil$,
where $\lceil x \rceil$ is the smallest integer $m \ge x$.
If $X$ is not a $\be$-set with $\lceil \be^{-1} \rceil< \lceil \al^{-1} \rceil$,
then $d(X) = 2 \lceil \al^{-1} \rceil$. For instance,
we have $$d(\ol \Ga^{1/m}_n(r,h)) = 2m$$
for
the closure $\ol \Ga^{1/m}_n(r,h)$ of $\Ga^{1/m}_n(r,h)$, where $m \in \N_{\ge 1}$.

\subsection{Results}
\subsubsection{Smooth plot-analytic functions}
The first result extends \Cref{res:A} to simple closed UPC sets 
with a precise control of the degree of the required polynomial $1$-plots in terms of $d(X)$.

\begin{theorem} \label{thm:UPC}
   Let $X \subseteq \R^n$ be a simple closed UPC set and let $d := d(X)$. Then
\[
    \sC^\om_{1,d}(X) \cap \cC^\infty|_X = \cC^\om|_X.
\]
\end{theorem}

We shall see in \Cref{ex:1} that $d = d(X)$ is optimal:
in general, 
$\sC^\om_{1,d}(X) \cap \cC^\infty|_X \not\subseteq \cC^\om|_X$ if $d < d(X)$.

If we do not restrict the degree of the polynomial plots, we may infer:

\begin{corollary} \label{cor:UPC}
   Let $X \subseteq \R^n$ be a simple closed set such that for each $z \in \p X$ there exists 
   a closed UPC set $X_z$ with $z \in X_z \subseteq X$.
   Then
   \[
    \sC^\om_{1}(X) \cap \cC^\infty|_X = \cC^\om|_X.
   \]
\end{corollary}

For instance, the corollary applies to the compact subset $X = X' \cup \bigcup_{m =1}^\infty X_m$ of $\R^2$, 
where $X_m := \overline{\Ga}^{1/m}_2(\frac{1}{m^2},\frac{1}{m}) + (2\sum_{\ell =1}^{m}\frac{1}{\ell^2}- \frac{1}{m^2}, -\frac{1}{m})$
and $X' = [0, \tfrac{\pi^2}3] \times [0,1]$. It is simple, since the cusps $X_m$ converge to the point $(\tfrac{\pi^2}3,0)$.
The spikes $X_m$ can be replaced by suitable horn-like sets.

\subsubsection{Plot-analytic functions that are not presupposed to be smooth}
Next we will discuss extensions of \Cref{res:B} to closed fat sets. In the following, we will not assume that $f$ is smooth.

By a \emph{simplex} in $\R^n$ we mean the convex hull of any collection of $n+1$ affinely independent points in $\R^n$.

\begin{theorem} \label{thm:Lipschitz}
   Let $X \subseteq \R^n$, $n \ge 2$, be a simple closed set such that for each $z \in \p X$ there exists 
   a simplex $X_z$ with $z \in X_z \subseteq X$.
Then
\[
    \sC^\om_{2,2}(X) = \cC^\om|_X.
\]
\end{theorem}

This is best possible; see \Cref{ex:2}.
The assumption is fulfilled for all Lipschitz sets. 
At this stage, we do not know if there is an analogue of \Cref{thm:Lipschitz} 
for H\"older sets or simple fat closed subanalytic sets.
But in dimension two, we can give a fairly complete answer.
Here (for technical reasons) we need to work with 
\begin{equation*}
    d'(X) := 2 \min_{(m,D) \in \on{char}(X)} m \cdot D. 
\end{equation*}
For all compact H\"older sets $X$ we have $d'(X) = d(X)$;
in general, $d'(X)\ge d(X)$.

\begin{theorem} \label{thm:2dim}
   Let $X \subseteq \R^2$ be a simple compact fat subanalytic or H\"older set and let $d':= d'(X)$. 
Then 
\[
    \sC^\om_{2,d'}(X)= \cC^\om|_X.
\]
\end{theorem}

It is likely that in this statement $d' = d'(X)$ is not optimal.

\begin{corollary} \label{cor:2dim}
   Let $X \subseteq \R^2$ be a simple closed set such that for each $z \in \p X$ there exists 
   a compact fat subanalytic set $X_z$ with $z \in X_z \subseteq X$.
   Then 
\[
    \sC^\om_{2}(X)= \cC^\om|_X.
\]
\end{corollary}

Note that for each H\"older set $X'_z$ with $z \in X'_z \subseteq X$ we may find a 
compact fat subanalytic set $X_z$ with $z \in X_z \subseteq X'_z$; the converse is not always possible.

\subsubsection{Curve-analytic functions}
We give an application to curve-analytic functions on closed fat sets 
which extends a recent result of Kucharz and Kurdyka \cite{Kucharz:2023aa}, see \Cref{res:C} below.

Let $X\subseteq \R^n$ be nonempty and $f : X \to \R$ a function.
We say that $f$ is \emph{curve-analytic} if for each real analytic arc $c : (-\ep,\ep) \to \R^n$ with 
image contained in $X$ 
there exist $\de \in (0,\ep]$ and a real analytic function $F : U \to \R$ defined on an open neighborhood $U$ of $c(0)$ in $\R^n$
such that $c(t) \in U$ and  $F(c(t)) = f(c(t))$ for all $t \in (-\de,\de)$.

\begin{remark}
   If $X$ is a real analytic manifold, then $f : X \to \R$ is curve-analytic if and only if 
   $f|_C$ is real analytic 
for every locally irreducible real analytic set $C$ of dimension $1$ in $X$; see \cite[Lemma 2.2]{Kucharz:2023aa}.
(That $f|_C$ is real analytic means that for each $x \in C$ there is a neighborhood $U$ of $x$ in $X$ and a real analytic function $F : U \to \R$
such that $f|_{C\cap U} = F|_{C \cap U}$.)
\end{remark}

\begin{theoremA}[{\cite{Kucharz:2023aa}}] \label{res:C}
A subanalytic function $f : X \to \R$ on a real analytic manifold $X$ 
is real analytic if and only if it is curve-analytic. 
\end{theoremA}

It is an open question, if the assumption that $f$ is subanalytic is necessary.

By definition, if $f : X \to \R$ is curve-analytic, then $f$ is arc-analytic, in particular, $f \in \sC^\om_1(X)$. 
Thus, any \emph{smooth} curve-analytic $f : X \to \R$ 
is the restriction of a real analytic function, for all $X$ satisfying the assumptions of \Cref{cor:UPC}. 
More interestingly, without presupposing that $f$ is smooth, we have:   

\begin{theorem} \label{thm:curve-analytic}
   Let $X\subseteq \R^n$ be a simple closed set 
   such that one of the following conditions is satisfied:
   \begin{enumerate}
       \item $n$ is arbitrary and for each $z \in \p X$ there is a simplex $X_z$ with $z \in X_z \subseteq X$.
       \item $n=2$ and for each $z \in \p X$ there is a compact fat subanalytic set $X_z$ with $z \in X_z \subseteq X$.
   \end{enumerate}
   Let $f : X \to \R$ be any function
   such that $f|_{X_z}$ is subanalytic for all 
   $X_z$ appearing in \thetag{1} or \thetag{2}.
   Then $f$ is curve-analytic if and only if $f \in \cC^\om|_X$.
\end{theorem}

All the results are proved in \Cref{sec:proofs}. 
The examples in \Cref{sec:examples} show that the connection between the degree of polynomial plots
that recognize real analyticity and the regularity of the boundary is optimal
and complement the investigation.

\section{Proofs} \label{sec:proofs}

\subsection{Localization of the problems}

\Cref{prop:gluing} shows that it suffices to study the 
problems locally at boundary points. First we recall a lemma from \cite{Rainer18}.

\begin{lemma}[{\cite[Lemma 6.1]{Rainer18}}] \label{lem:glue}
  Let $X \subseteq \R^n$ be closed and $U \subseteq \R^n$ open with $U \cap X \ne \emptyset$. 
  Then there exists an open subset $U_0 \subseteq U$ with $U_0 \cap X = U \cap X$ 
  such that 
  for all $x \in U_0$ and all $a \in X$ that realize the distance of $x$ to $X$, i.e., $\on{dist}(x,a) = \on{dist}(x,X)$,
  the line segment $[x,a]$ is contained in $U_0$.
\end{lemma}

\begin{proposition} \label{prop:gluing}
    Let $X\subseteq \R^n$ be a closed fat set.
    Let $f : X^\o \to \R$ be real analytic.
    Suppose that for all $x \in \p X$ there exists a neighborhood $U_x$ of $x$ in $\R^n$ and 
    a real analytic function $F_x : U_x \to \R$ such that $F_x|_{U_x \cap X^\o} = f|_{U_x \cap X^\o}$.
    Then there exists a real analytic extension of $f$ to an open neighborhood of $X$.
\end{proposition}

\begin{proof}
    Let us show that $f$ and the $F_x$ glue coherently to a global real analytic extension of $f$.
  Invoking \Cref{lem:glue}, we replace $U_x$ by the connected component of $(U_x)_0$ that contains $x$. 
  Then the open cover $\{U_x : x \in \p X\}\cup \{X^\o\}$ of $X$ has the property that for each $z \in U_x$ 
  and each $a \in X$ that realizes the distance of $z$ to $X$ 
  the segment $[z,a]$ lies in $U_x$. 
  Let $V$ be a connected component of $U_x \cap U_y$.
  For each $z \in V$ and each $a \in X$ that realizes the distance of $z$ to $X$
  the line segment $[z,a]$ lies in $V$. 
  Since $X$ is fat, $V$ has nonempty intersection with $X^\o$,
  and on this intersection $F_x$ and $F_y$ coincide with $f$, by assumption.
  By the identity theorem, $F_x$ and $F_y$ coincide on $V$ and hence on $U_x \cap U_y$, since the connected component $V$ was arbitrary.
\end{proof}

\subsection{Strongly injective homomorphisms}

We will use the following result.

\begin{theorem}[{\cite{EakinHarris77, Gabrielov73}}] \label{thm:EakinHarris}
        Let $\Ph : \K\{X_1,\ldots,X_n\} \to \K\{Y_1,\ldots,Y_k\}$ be a homomorphism 
        of convergent power series rings over the field $\K$.
        Then the following conditions are equivalent:
        \begin{enumerate}
            \item $\Ph$ is \emph{strongly injective}, i.e., if $\hat \Ph$ denotes the natural extension of $\Ph$ to formal power series, 
            then $\hat \Ph(f) \in \K\{Y_1,\ldots,Y_k\}$ implies $f \in \K\{X_1,\ldots,X_n\}$. 
            \item The generic rank of $\Ph$ is $n$.
        \end{enumerate}
\end{theorem}

\subsection{Proof of \Cref{thm:UPC} and \Cref{cor:UPC}}

\begin{lemma} \label{lem:1tom}
   Let $X \subseteq \R^n$ be nonempty and closed, $f \in \cC^\infty|_X$, and $d$ a positive integer.
   Then the following conditions are equivalent.
   \begin{enumerate}
       \item $f \in \sC^\om_{1,d}(X)$.
       \item $f \in \sC^\om_{m,d}(X)$ for all $m \ge 1$.
       \item $f \in \sC^\om_{m,d}(X)$ for some $m \ge 1$.
   \end{enumerate} 
\end{lemma}

\begin{proof}
    (1) $\Rightarrow$ (2)
    Let $f \in \sC^\om_{1,d}(X)$, $m \ge 1$, and $p \in \cP_{m,d}(X)$. 
    Fix a representative of $p$, also denoted by $p$.
    So there exists an open neighborhood $U$ of $0 \in \R^m$ such that $p(U) \subseteq X$.
    Let $x \in U$, $v \in \mathbb S^{n-1}$, and $q(t) := x +t v$ for $t$ near $0 \in \R$ such that $q(t) \in U$.
    Then the germ of $q$ belongs to $\cP_{1,1}(U)$ and that of $p \o q$ to $\cP_{1,d}(X)$. 
    Thus, $f \o p \o q$ is real analytic, by (1). 
    Since $f \o p$ is smooth,
    $f \o p$ is real analytic, by \Cref{res:A}.
    (In fact, it is clear from the above that $f \o p|_{\ell \cap U}$ is real analytic, where $\ell$ is the affine line 
    generated by $q$.)

    (2) $\Rightarrow$ (3)
    Trivial.

    (3) $\Rightarrow$ (1)
This follows easily from the fact that $\cP_{1,d}(X)$ can be identified with a subset of $\cP_{m,d}(X)$ 
    (viewing $x_1 \mapsto p(x_1)$ as $(x_1,\ldots,x_m) \mapsto p(x_1)$).
\end{proof}

\begin{proof}[Proof of \Cref{thm:UPC}]
    The inclusion $\cC^\om|_X \subseteq \sC^\om_{1,d}(X) \cap \cC^\infty|_X$ is clear.
    Suppose that $f \in \sC^\om_{1,d}(X) \cap \cC^\infty|_X$. 
    By \Cref{res:A}, $f|_{X^\o}$ is real analytic.
    By \Cref{prop:gluing}, 
it suffices to show that 
for each $x \in \p X$ 
    there exist a fat set $\Pi_x$ such that $x \in \p \Pi_x$ and $\Pi_x^\o \subseteq X^\o$, 
    a neighborhood $U_x$ of $x$ in $\R^n$, and a real analytic function $F_x : U_x \to \R$ 
    such that $f|_{U_x \cap \Pi_x^\o}= F_x|_{U_x \cap \Pi_x^\o}$.
Indeed, since $X$ is simple, we may assume (after possibly shrinking $U_x$) that $U_x \cap X^\o$ 
    is connected. 
    Thus, $f|_{U_x \cap X^\o}= F_x|_{U_x \cap X^\o}$, by the identity theorem, and
    \Cref{prop:gluing} implies the statement.

Fix $x \in \p X$. Since $X$ is a UPC set with $d = d(X)$, 
we have $d = 2 \max\{m,D\}$ for some $(m,D) \in \on{char}(X)$. 
So there is a polynomial curve $h_x : \R \to \R^n$ of degree at most $D$ 
such that 
$h_x(0)=0$ and \label{page:proof11}
\[
    \on{dist}(h_x(t),\R^n \setminus X) \ge M\, t^m, \quad x \in X,\, t \in [0,1]. 
    \]
We have $h_x(t) = t^p \tilde h_x(t)$, for some integer $p\ge 1$, where $\tilde h_x(0) \ne  0$. 
Let $v_1 := \tilde h_x(0)/|\tilde h_x(0)|$ and choose unit vectors $v_2,\ldots,v_n$ in $\R^n$ 
such that $v_1,\ldots, v_n$ forms a basis of $\R^n$.
Then 
\[
    \th_x(s_1,\ldots,s_n) := h_x(s_1) + s_2^m v_2 + \cdots + s_n^m v_n
\]
is a polynomial map $\th_x : \R^n \to \R^n$ of degree at most $\max\{m,D\}$ which takes the set 
$C := \{(s_1,\ldots,s_2) : 0 \le s_1 \le 1, \, |s_i| \le \big(\tfrac{M}{2(n-1)}\big)^{1/m} s_1 \text{ for } i\ge 2\}$ into $X$. Indeed,
\begin{align*}
    \on{dist}(\th_x(s_1,\ldots,s_n), \R^n \setminus X) 
    &\ge \on{dist}(h_x(s_1), \R^n\setminus X) - |s_2|^m  - \cdots - |s_n|^m 
    \\
    &\ge M s_1^m - \sum_{j=2}^n  \tfrac{M}{2(n-1)} s_1^m
    \\
    &= \tfrac{M}{2} s_1^m, 
\end{align*}
if $(s_1,\ldots,s_n) \in C$.
Since $C$ is a $1$-set and $0 \in \p C$, 
there is a linear isomorphism $\ell : \R^n \to \R^n$ 
such that $\ell (S)\subseteq C$, where $S$ is the convex hull of $0,e_1,\ldots,e_n$.
There is an open neighborhood $U$ of $0$ in $\R^n$ such that 
$q : \R^n \to \R^n$ given by 
\[
    q(x_1,\ldots,x_n) := (x_1^2,\ldots,x_n^2)
\]
takes $U$ into $S$.
Thus the polynomial map $\ps_x := \th_x \o \ell \o q : \R^n \to \R^n$ of degree at most $d$ takes $U$ into $X$ 
and so we may consider the composite $f \o \ps_x|_U$. 
By \Cref{lem:1tom}, $f \o \ps_x|_U$ is real analytic. 
Since $f$ is smooth, by assumption, we have its formal Taylor series $F_x = T_x f$ at $x$. 
    Then the power series $F_x \o \ps_x = T_0 (f \o \ps_x)$ has positive radius of convergence. 
    By \Cref{thm:EakinHarris},
    $F_x$ has positive radius of convergence and thus defines a real analytic function, again denoted by $F_x$, 
    on an open neighborhood $U_x$ of $x$ in $\R^n$. 
    Clearly,
$f|_{U_x \cap \Pi_x}= F_x|_{U_x \cap \Pi_x}$, where $\Pi_x := \ps_x(U)$.
This shows the claim and hence completes the proof. 
\end{proof} 

\begin{proof}[Proof of \Cref{cor:UPC}] 
This follows easily from \Cref{thm:UPC} (applied to each $X_z$) and \Cref{prop:gluing}.
Note that $X$ is fat.
\end{proof} 

\subsection{Proof of \Cref{thm:Lipschitz}}
We need a variant of \Cref{lem:1tom}, where $f : X \to \R$ is not necessarily smooth.

\begin{lemma} \label{lem:2tom}
   Let $X \subseteq \R^n$ be nonempty, $f : X \to \R$ any function, and $d$ a positive integer.
   Then the following conditions are equivalent.
   \begin{enumerate}
       \item $f \in \sC^\om_{2,d}(X)$.
       \item $f \in \sC^\om_{m,d}(X)$ for all $m \ge 2$.
       \item $f \in \sC^\om_{m,d}(X)$ for some $m \ge 2$.
   \end{enumerate} 
\end{lemma}

\begin{proof}
    (1) $\Rightarrow$ (2)
    The proof is analogous to the one of \Cref{lem:1tom}
    except that now $q$ is a $(2,1)$-plot that 
     parameterizes a piece of an affine $2$-plane contained in $U$
     and
     we use \Cref{res:B}.

    (2) $\Rightarrow$ (3)
    Trivial.

    (3) $\Rightarrow$ (1)
    $\cP_{2,d}(X)$ can be seen as a subset of  
    $\cP_{m,d}(X)$ if $m\ge 2$. 
\end{proof}

\begin{lemma} \label{lem:invariant1}
    Let $q : \R^n \to \R^n$ be the map defined by $q(x_1,\ldots,x_n) := (x_1^2,\ldots,x_n^2)$.
    Let $I_r := (-r,r)^n \subseteq \R^n$. 
    Then $Q_r := q(I_r)$ is a neighborhood of $0$ in the first orthant 
    $Q := \{x \in \R^n : x_i \ge 0 \text{ for all } 1 \le i \le n\}$.  
    For any function $f: Q_r \to \R$ such that $f \o q : I_r \to \R$ is real analytic 
    there exists a real analytic function $F : U \to \R$ defined in a neighborhood $U$ of $0$ in $\R^n$
    such that $F|_{U \cap Q_r} = f|_{U \cap Q_r}$.
\end{lemma}

\begin{proof}
    This is a consequence of a theorem of Luna \cite{Luna76}, 
since $x_1^2,\ldots,x_d^2$ are the basic invariants of the natural action of the group $G = \Z_2 \times \cdots \times \Z_2$ on 
$\R \times \cdots \times \R$.

Let us sketch an alternative proof. 
Take a smooth function $\ch : \R^n \to \R$ such that $\ch \equiv 1$ in a neighborhood of $0$ and $\ch \equiv 0$ 
outside $I_s$, where $s< r^2$. Then $\tilde f := \ch f$ extends by zero to a function defined on all of $Q = q(\R^n)$
and $\tilde f \o q : \R^n \to \R$ is a $G$-invariant smooth function.
Thus, by \cite{Glaeser63F} or \cite{Schwarz75}, there is a smooth function $g : \R^n \to \R$ such that $\tilde f \o q = g \o q$, 
i.e., $g$ extends $\tilde f$ to $\R^n$.
In particular, $f$ has a smooth extension to a neighborhood of $0$ in $\R^n$. 
If $F$ denotes the Taylor series at $0$ of $f$ (equivalently, of $g$),  
then $F \o q = T_0(f \o q)$ has positive radius of convergence, since $f\o q$ is real analytic. 
By \Cref{thm:EakinHarris}, $F$ converges and defines a real analytic function with the desired properties.
\end{proof}

\begin{proof}[Proof of \Cref{thm:Lipschitz}]
    We have to show $\sC^\om_{2,2}(X) \subseteq \cC^\om|_X$.
    Let $f \in \sC^\om_{2,2}(X)$. By \Cref{lem:2tom},   $f \in \sC^\om_{n,2}(X)$.
    Fix $z \in \p X$ and let $X_z$ be a simplex such that $z \in X_z \subseteq X$.
    There is an affine isomorphism $\ell : \R^n \to \R^n$ with $\ell(0) =z$ such that 
    $X_z$ is the image of the convex hull of $0,e_1,\ldots,e_n$ under $\ell$.
    So there is an open neighborhood $U$ of $0 \in \R^n$ such that $f \o \ell \o q : U \to \R$ 
    is well-defined and real analytic. 
    It follows from \Cref{lem:invariant1} that $f \o \ell|_{q(U)}$ has a real analytic extension to some neighborhood of $0$
    and thus $f|_{\ell(q(U))}$ has a real analytic extension to some neighborhood of $z$.
    Since $z$ was arbitrary, the theorem now follows from \Cref{prop:gluing}.
\end{proof}

\subsection{Proof of \Cref{thm:2dim} and \Cref{cor:2dim}}

Let $d \ge 3$ and
consider the map $\si = (\si_1,\si_2) : \R^2 \to \R^2$ defined by 
\[
    \si_1(x,y) := x^2 + y^2, \quad \si_2(x,y) := \Re((x+iy)^d).
\]
Then $\si(\R^2) = \{(x,y) \in \R^2 : x \ge 0,\, |y| \le x^{d/2}\} =:\Si$. 
In fact, in polar coordinates $\si_1(r e^{i\th}) = r^2$ and $\si_2(r e^{i\th}) = r^d \cos(d \th)$
so that $\si$ maps the circle with radius $r$ around $0$ onto the vertical line segment between the points $(r^2,\pm r^d)$.
The polynomials $\si_1$ and $\si_2$ are the basic invariants of the dihedral group $G=I^d_2$ 
consisting of the orthogonal transformations of $\R^2$ that preserve the regular $d$-gon.

\begin{lemma} \label{lem:invariant2}
    Let $I_r = (-r,r)^2 \subseteq \R^2$. 
    Then $\Si_r := \si(I_r)$ is a neighborhood of $0$ in $\Si$.  
    For any function $f: \Si_r \to \R$ such that $f \o \si : I_r \to \R$ is real analytic 
    there exists a real analytic function $F : U \to \R$ defined in a neighborhood $U$ of $0$ in $\R^2$
    such that $F|_{U \cap \Si_r} = f|_{U \cap \Si_r}$.
\end{lemma}

\begin{proof}
    Follow the proof of \Cref{lem:invariant1} and make the obvious modifications.  
\end{proof}

\begin{proof}[Proof of \Cref{thm:2dim}]
We first assume that $X \subseteq \R^2$ is a compact H\"older set.
Then $d:= d(X) =d'(X)$ and 
for each $x \in \p X$ 
there is an affine isomorphism $\ell : \R^2 \to \R^2$ with $\ell(0)=x$ 
and $r>0$ such that $\ell(\Si_r) \subseteq X$.
If $f \in \sC^\om_{2,d}(X)$,
then $f \o \si : I_r \to \R$ 
is well-defined and real analytic. 
Now the result follows from \Cref{lem:invariant2} and \Cref{prop:gluing}.

Let now  
$X \subseteq \R^2$ be a simple compact fat 
subanalytic set and $d' =d'(X)$.
Let $f \in \sC^\om_{2,d'}(X)$. 
Fix $x \in \p X$. There exist $(m,D) \in \on{char}(X)$ with $d' = 2 mD$, 
a polynomial curve $h_x$ of degree at most $D$, and a basis $v_1, v_2$ as in the proof of \Cref{thm:UPC} (on page \pageref{page:proof11}). 
Consider the polynomial map $\Th_x : \R^2 \to \R^2$ given by 
\[
    \Th_x(s_1,s_2) := h_x(s_1) + s_2 v_2.
\]
It has degree at most $D$ and takes the cusp $S =\{(s_1,s_2) : 0 \le s_1 \le 1,\, |s_2| \le \tfrac{M}2 s_1^m \}$
into $X^\o \cup \{x\}$.
The function $f \o \Th_x|_S$ belongs to $\sC^{\om}_{2,2m}(S)$ and hence it is the restriction of a real analytic 
function defined on a neighborhood of $S$, by \Cref{thm:2dim} for H\"older sets (since $d'(S)=d(S)=2m$).
We may compute 
\[
\p_{s_2}^k (f \o \Th_x)(s_1,s_2) = d^k_{v_2} f(\Th_x(s_1,s_2)), \quad k \in \N,
\]
for $(s_1,s_2) \in S^\o$ (since $f$ is real analytic in $X^\o$, by \Cref{res:B}) and 
letting $(s_1,s_2) \to (0,0)$ we see 
that the directional derivatives on the 
right-hand side extend continuously to $x$. 
This remains true for all $v_2$ in a small neighborhood of the chosen $v_2$. 
By the proof of \cite[Theorem 1.14]{Rainer18} or \cite[Theorem C]{Rainer:2021tr}, 
we find that $f$ is the restriction to $X$ of a $\cC^\infty$-function defined on $\R^2$.
The cited proof shows that the derivatives of all orders 
extend continuously to the boundary of $X$. Since compact fat subanalytic sets are Whitney $p$-regular, for some positive integer $p$,
these derivatives define a Whitney jet of class $\cC^\infty$ on $X$ which extends to a smooth function on $\R^2$.

Now it suffices to invoke \Cref{thm:UPC} and \Cref{lem:1tom} (since $d'(X) \ge d(X)$) 
in order to complete the proof of \Cref{thm:2dim}.
\end{proof}

\Cref{cor:2dim} follows easily from 
\Cref{thm:2dim} and \Cref{prop:gluing}.

\subsection{Proof of \Cref{thm:curve-analytic}}

If $f$ is the restriction of a real analytic function, then clearly $f$ is curve-analytic.
So suppose that $f$ is curve-analytic.

\begin{lemma} \label{lem:KK}
    Let $m\ge 1$.
    Let $p \in \cP_m(X)$ be such that $f \o p$ is subanalytic 
    (for instance, if the image of $p$ is contained in $X_z$ for some $z$).
    Then $f \o p$ is real analytic. 
\end{lemma}

\begin{proof}
    Let $c : (-\ep,\ep) \to \R^m$ be a real analytic arc in the domain of definition of (a representative of) $p$. 
    Since $f$ is curve-analytic, there exist $\de \in (0,\ep]$ and a real analytic function $F : U \to \R$
    on a neighborhood $U$ of $p(c(0))$ in $\R^n$ 
    such that $p(c(t)) \in U$ and $F(p(c(t))) = f(p(c(t)))$ for all $t \in (-\de,\de)$. 
    Thus $f \o p$ is curve-analytic and subanalytic so that $f\o p$ is real analytic, by \Cref{res:C}.
\end{proof}

By \Cref{lem:KK}, $f|_{X_z} \in \sC^\om_{2}(X_z)$ and hence $f|_{X_z} \in \cC^\om|_{X_z}$, 
thanks to \Cref{thm:Lipschitz} and \Cref{thm:2dim}. Then \Cref{thm:curve-analytic} follows from \Cref{prop:gluing}.

\begin{remark}
    An important step in the proof of \Cref{res:C} in \cite{Kucharz:2023aa}
    is the criterion that an arc-analytic subanalytic function $f$ defined on an open subset $U$ of $\R^n$ 
    is real analytic at a point $x \in U$ 
    if and only if $f$ satisfies the \emph{$(p,q)$-test at $x$} for all pairs $(p,q)$ of positive integers.
    That means that for every pair of linearly independent vectors $\xi,\nu \in \R^n$ and all $a,b \in \R$,
    $\vh(t) := f(x + a t^p \xi + b t^q \nu)$ is a real analytic germ at $0 \in \R$ and 
    the support of $\vh$ is contained in $\N p + \N q$. 
\Cref{thm:curve-analytic} suggests that often not all $(p,q)$-tests are necessary to detect real analyticity.
\end{remark}

\section{Examples} \label{sec:examples}

The first example shows that 
in general the degree of polynomial test plots in \Cref{thm:UPC} is optimal.

\begin{example} \label{ex:1}
    Let $D$ be a positive integer.
  The horn $H = \{(x,y) \in \R^2 : 0 \le x \le 1, \, x^{D} \le y \le 2 x^D\}$ 
  is a UPC set with characteristic $(D,D)$ and $d(H)  = 2D$. 
  Indeed, every polynomial curve $h : \R \to \R^2$ with $h(0) = 0$ 
  and $h([0,1]) \subseteq H$ must be of degree at least $D$ and $\on{dist}(h(t),\R^2\setminus H)\lesssim t^D$ for $t \in [0,1]$. 
  Every polynomial $1$-plot in $H$ through $0$ must be of degree at least $2D$.
  
  The function $f : H \to \R$ given by $f(x,y) := e^{-1/(x^2 +y^2)}$ for $(x,y) \ne (0,0)$ and $f(0,0) := 0$ 
  belongs to $\cC^\infty|_H$ and to $\sC^\om_{1,2D-1}(H)$, simply because 
  $f$ is real analytic away from $0$ and for all $p \in \cP_{1,2D-1}(H)$ we have $p(0)  \ne 0$.

  On the other hand, $f \not\in \sC^{\om}_{1,2D}(H)$ either by \Cref{thm:UPC} or by 
observing that for $p(t):= (t^{2},t^{2D})$, which lies in $H$ for all $t$ near $0$, the composite
\[
(f\o p)(t) = \exp\Big(-\frac{1}{t^{4}(1+t^{4(D-1)}) }\Big) 
\]
has no analytic extension to $t=0$.
\end{example}

The following easy example shows that \Cref{thm:Lipschitz} is best possible.

\begin{example} \label{ex:2}
    Let $X = \{x \in \R^n : x_i \ge 0 \text{ for all } 1 \le i \le n\}$ be the first orthant.
    A plot $p \in \cP_{m,1}(X)$ ($m$ arbitrary) can satisfy $p(0) = 0$ only if $p\equiv 0$. 
    In other words, an arbitrary function $f \in \sC^\om_{m,1}(X)$ is unrestricted at $0$.
\end{example}

The next example shows that the assumption of being simple is essential.

\begin{example} \label{ex:simple}
  Consider the set $X_1 := \{(x,y) \in \R^2 : x\ge 0,\, x^{\sqrt 2} \le y \le x^{\sqrt 2}+x^2\}$. 
  Every smooth curve $c$ in $X_1$ must vanish to infinite order on $c^{-1}(0)$; see \cite[Example 6.7]{Rainer:2021tr}. 
  Thus $\sC^\om_m(X_1)$ ($m$ arbitrary) 
  contains elements that do not have a real analytic extension.

  Let $X_2 := \{(x,y) \in \R^2 : x\le 0\}$ be the left half-plane 
  and consider $X := X_1 \cup X_2$ which is not simple.
  But for each point $z$ of $\p X$ there exists a simple fat compact subanalytic set $X_z$ such that $z \in X_z \subseteq X$.
  Again it is easy to find (even smooth) elements of $\sC^\om_m(X)$ 
  that do not admit a real analytic extension.

  Note that it was wrongly stated in \cite[Corollary 1.17]{Rainer18} that, in a related result, the assumption 
  that the set is simple is not needed. 
\end{example}

There are algebraic sets $X$ in $\R^2$ for which $\sC^\om_{m}(X)$ ($m$ arbitrary) is essentially bigger than $\cC^\om|_X$: 

\begin{example}
    Consider the algebraic set $X = \{(x,y) \in \R^2 : x^3 = y^2\}$ and the function $\vh : X \to \R$ defined by $\vh(x,y) = y^{1/3}$. 
    Then $\vh \in \sC^\om_{m}(X)$ for all $m\ge 1$. Indeed, let $p=(p_1,p_2) \in \cP_m(X)$. The composite $u = \vh \o p$ 
    satisfies $u^2 = p_1$ and $u^3 = p_2$ so that $u$ is smooth, by a result of Joris \cite{Joris82}, and in turn real analytic, 
    by \cite[Proposition VI.3.11]{Malgrange67}. 
    
    In fact, we can completely describe $\sC^\om_{m}(X)$: for all $m \ge 1$, 
    \[
        \sC^\om_{m}(X) = \vh^* \cC^\om(\R) := \{g \o \vh : g \in \cC^\om(\R)\}.
    \]
    That $\vh^* \cC^\om(\R)$ is contained in $\sC^\om_m(X)$ is an easy consequence of the fact that $\vh \in \sC^\om_{m}(X)$.   
    Conversely, suppose that $f \in \sC^\om_m(X)$.
    Then $g(t) := f(t^2,t^3)$ defines an element $g \in \cC^\om(\R)$ satisfying $g(\vh(x,y)) = f(y^{2/3},y) = f(x,y)$.
\end{example}


\def\cprime{$'$}
\providecommand{\bysame}{\leavevmode\hbox to3em{\hrulefill}\thinspace}
\providecommand{\MR}{\relax\ifhmode\unskip\space\fi MR }
\providecommand{\MRhref}[2]{%
  \href{http://www.ams.org/mathscinet-getitem?mr=#1}{#2}
}
\providecommand{\href}[2]{#2}

\end{document}